\documentclass[11pt,reqno]{amsart}
\usepackage[utf8]{inputenc}
\usepackage{a4wide}
\usepackage{appendix}

\usepackage{csquotes}
\usepackage{graphicx}
\usepackage{enumerate}
\usepackage{hyperref}
\usepackage{cleveref}
\usepackage{amsmath,amsopn,amssymb,amsfonts,stmaryrd}
\usepackage{verbatim}
\usepackage{amsthm}
\usepackage{mathtools}
\usepackage{color}
\usepackage{enumitem}
\usepackage[framemethod=TikZ]{mdframed}
\usepackage{bbm}
\usepackage{mathrsfs}
\usepackage{booktabs}
\usepackage{caption}
\usepackage{bm}
\usepackage{tensor}

\makeatletter
\@namedef{subjclassname@2020}{%
	\textup{2020} Mathematics Subject Classification}
\makeatother

\theoremstyle{plain}
\newtheorem{thm}{Theorem}[section]
\newtheorem{cor}[thm]{Corollary}
\newtheorem{lem}[thm]{Lemma}
\newtheorem{prop}[thm]{Proposition}
\newtheorem{conj}[thm]{Conjecture}
\newtheorem{ques}[thm]{Question}
\newtheorem{fact}[thm]{Fact}
\newtheorem{facts}[thm]{Facts}
\newtheorem{rem}[thm]{Remark}

\theoremstyle{definition}
\newtheorem{example}[thm]{Example}
\newtheorem{defn}[thm]{Definition}

\usepackage{aliascnt}
\usepackage{cleveref}

\theoremstyle{plain} 
\newaliascnt{cor}{thm}
\newtheorem{cor}[cor]{Corollary}
\aliascntresetthe{cor}
\crefname{cor}{corollary}{corollaries}
\Crefname{cor}{Corollary}{Corollaries}

\newaliascnt{lem}{thm}
\newtheorem{lem}[lem]{Lemma}
\aliascntresetthe{lem}
\crefname{lem}{lemma}{lemmas}
\Crefname{lem}{Lemma}{Lemmas}

\newaliascnt{prop}{thm}
\newtheorem{prop}[prop]{Proposition}
\aliascntresetthe{prop}
\crefname{prop}{proposition}{propositions}
\Crefname{prop}{Proposition}{Propositions}

\newaliascnt{conj}{thm}
\newtheorem{conj}[conj]{Conjecture}
\aliascntresetthe{conj}
\crefname{conj}{conjecture}{conjectures}
\Crefname{conj}{Conjecture}{Conjectures}

\crefname{ques}{question}{questions}
\Crefname{ques}{Question}{Questions}

\newaliascnt{fact}{thm}
\newtheorem{fact}[fact]{Fact}
\aliascntresetthe{fact}
\crefname{fact}{fact}{facts}
\Crefname{fact}{Fact}{Facts}

\newaliascnt{facts}{thm}
\newtheorem{facts}[facts]{Facts}
\aliascntresetthe{facts}
\crefname{facts}{fact}{facts}
\Crefname{facts}{Fact}{Facts}

\newaliascnt{rem}{thm}
\newtheorem{rem}[rem]{Remark}
\aliascntresetthe{rem}
\crefname{rem}{remark}{remarks}
\Crefname{rem}{Remark}{Remarks}

\theoremstyle{definition} 
\newaliascnt{example}{thm}
\newtheorem{example}[example]{Example}
\aliascntresetthe{example}
\crefname{example}{example}{examples}
\Crefname{example}{Example}{Examples}

\newaliascnt{defn}{thm}
\newtheorem{defn}[defn]{Definition}
\aliascntresetthe{defn}
\crefname{defn}{definition}{definitions}
\Crefname{defn}{Definition}{Definitions}

\crefname{clai}{claim}{claims}
\Crefname{clai}{Claim}{Claims}

\numberwithin{equation}{section}

\def\G{\mathcal{G}}

\def\d{\delta}

\def\k{\kappa}

\def\d{\delta}

\def\k{\kappa}

\def\RR{{\mathbb R}}
\def\NN{{\mathbb N}}

\def\d{{\mathrm{d}}}

\def\id{\mathrm{id}}

\def\SS{{\mathbb{S}}}

\def\bar{\overline}

\setlist[itemize]{noitemsep, topsep=0pt}

\makeatletter
\newcommand{\vast}{\bBigg@{2}}
\newcommand{\Vast}{\bBigg@{5}}
\makeatother

\newcommand{\RNum}[1]{\uppercase\expandafter{\romannumeral #1\relax}}
\allowdisplaybreaks

\title[Tonelli Lagrangians on Half Lie-Groups]{On Mañé's Critical Value for\\ Tonelli Lagrangians on Half Lie-Groups}

\author{Levin Maier}
\address{Faculty of Mathematics and Computer Science,
	University of Heidelberg,
	Im Neuenheimer Field 205,
	69120 Heidelberg, Germany}
\email{lmaier@mathi.uni-heidelberg.de}

\author{Francesco Ruscelli}
\address{Faculty of Mathematics and Computer Science,
	University of Heidelberg,
	Im Neuenheimer Field 205,
	69120 Heidelberg, Germany}
\email{fruscelli@mathi.uni-heidelberg.de}


\subjclass[2020]{}

\begin{document}
	\maketitle
\renewcommand{\abstractname}{Abstract}
	\begin{abstract}
In this article, we introduce \emph{Tonelli Lagrangians} on half-Lie groups equipped with a strong right-invariant Riemannian metric. 
These are right-invariant Lagrangians defined on the tangent bundle of a half-Lie group with quadratic growth on each fiber.
The main examples of half-Lie groups are groups of \( H^s \) or \( C^k \) diffeomorphisms of compact manifolds. \\[0.3em]
We show that the Euler–Lagrange flow exists globally. We then introduce three thresholds of the energy, called the Mañé critical values, and prove that under mild regularity and completeness assumptions on the half-Lie group any two points can be connected by a global Tonelli minimizer above the lowest of these energy thresholds. Under an additional assumption on the Lagrangian, such a minimizer is a flow line of the Euler–Lagrange flow.\\[0.3em]
This extends the work of Contreras from closed finite-dimensional manifolds to the infinite-dimensional context. 
Moreover, our results also extend the recent work of Bauer, Harms, and Michor from geodesic flows to Euler–Lagrange flows of Tonelli Lagrangians.\\[0.3em]
As an application, we obtain global well-posedness of all Euler–Poincaré equations associated with Tonelli Lagrangians on half-Lie groups equipped with strong right-invariant Riemannian metrics.
	\end{abstract}	
	\section{Introduction}\label{s:introduction}
\textbf{Infinite-dimensional Hamiltonian systems} trace back at least to the birth of modern quantum mechanics, specifically to E.~Schrödinger’s formulation~\cite{schrodinger1926quantisierung1} of wave mechanics via the celebrated Schrödinger equation. \\
Since then, many prominent partial differential equations have been recognized to admit a formulation as infinite-dimensional Hamiltonian systems—for instance, the Korteweg–de Vries equation, the nonlinear Schrödinger equation, the Euler equations from incompressible hydrodynamics and the Burgers equation, as well as the sine-Gordon, Camassa–Holm, and Benjamin–Ono equations, and certain equations in optimal transport. See~\cite{AK98, Khesin-Mis-Mod-inf-Newton} and the references therein for further details.\\
This insight has motivated a significant body of work devoted to studying the theoretical properties of infinite-dimensional Hamiltonian systems. In this work, we focus on Hamiltonian systems on \emph{half-Lie groups}. \\

\textbf{Half-Lie groups} are smooth manifolds and topological groups for which right translations are smooth, while left translations are only required to be continuous. Examples in which left translations are only continuous exist only in infinite dimensions. 

Their study is motivated by the fact that such groups arise naturally in the context of Arnold’s geometric formulation of mathematical hydrodynamics~\cite{Arnold66}. More precisely, Arnold showed that the Euler equations of hydrodynamics, which govern the motion of an incompressible and inviscid fluid in a fixed domain (with or without boundary), can be interpreted as the geodesic equations on the group of volume-preserving diffeomorphisms endowed with a right-invariant Riemannian metric. 

However, the group of volume-preserving diffeomorphisms is a tame Fréchet–Lie group, which leads to significant analytical difficulties. Following the approach of Ebin and Marsden~\cite{EM70}, one can instead work with $C^k$- or $H^s$-diffeomorphism groups, as first studied by Eells, Eliasson, and Palais~\cite{Eells66, Eliasson67, Palais68}, which provide the main examples of half-Lie groups.

Before moving on, we recall that in finite dimensions many Hamiltonian systems of physical interest arise as the Legendre duals of Lagrangian systems—for example, the motion of charged particles in a magnetic field and under potential forces. We now aim to introduce, in the infinite-dimensional setting, a natural class of Lagrangians to which the aforementioned electromagnetic systems naturally belong.\\

\textbf{Right-invariant Tonelli Lagrangians on half-Lie groups.} 
We first fix some notation. For a half-Lie group \( G \), we denote by \( e \) its neutral element and equip it with a \emph{strong}, \( G \)-right-invariant Riemannian metric \( \mathcal{G} \). For a precise definition of this notion, see \Cref{Def: strong right invariant Riemannian metrics}. 

From now on, we denote by \( (G, \mathcal{G}) \) a half-Lie group equipped with a strong Riemannian metric \( \mathcal{G} \). 
We call a $G$-right-invariant $C^2$-Lagrangian  
\[
L \colon TG \to \mathbb{R}
\]
\emph{Tonelli} on \( (G, \mathcal{G}) \) if the map \( v \mapsto L(e,v) \) is uniformly convex and grows quadratically in the fibers.
If, in addition, $\d_x L$ grows at most quadratically in $v$ within a neighborhood of the identity, we call $L$ \emph{strongly Tonelli}. 
For a precise meaning, we refer to \Cref{Def: rightv inv Tonelli Langrangian}.
\begin{rem}\label{rem: intro magnetic Langrang are Tonelli int}
A natural subclass of the space of Tonelli Lagrangians on $(G,\mathcal{G})$ is given by the electromagnetic Lagrangians, that is, Lagrangians of the form  
\begin{equation}\label{eq: electromagnetic lagrangians int}
    L(x,v) = \tfrac{1}{2}\,\Vert v\Vert^2_{\mathcal{G}} - \theta_x(v) + V(x)\, ,
\end{equation}
where $\theta \in \Omega^1(G)$ and $V \in C^{\infty}(G,\RR)$ are $G$-right-invariant. For $\theta = 0$ and $V = 0$, we simply recover the kinetic Lagrangian, whose flow lines are precisely the geodesics of $(G,\mathcal{G})$.
\end{rem}

The Tonelli assumptions imply that the Euler–Lagrange equation of \( L \), which in local coordinates can be written as
\begin{equation}\label{eq: ELE}
    \frac{\mathrm{d}}{\mathrm{d}t} \big(\partial_v L\big)_{(\gamma, \dot{\gamma})}
    = \big(\partial_x L\big)_{(\gamma, \dot{\gamma})},
\end{equation} 
is well-posed and defines a smooth flow \( \Phi_L \) on \( TG \), called the \emph{Euler–Lagrange flow} of \( L \) (see \Cref{lemm: ele define EL-flow}). 
These flow lines preserve the \emph{energy} of the system,
\begin{equation}\label{eq:energy}
    E \colon TM \to \mathbb{R}, \quad (x, v) \longmapsto (D_v L)_{(x, v)}(v) - L(x, v).
\end{equation}

Furthermore, in the case of a strong Tonelli Lagrangian $L$, the solutions of \eqref{eq: ELE} with energy \( E = \kappa \) are precisely the critical points of the \emph{time-free Lagrangian action functional} associated with \( L \) and the energy level \( \kappa \), which is given by
\begin{equation}\label{eq: def free action functional}
    \mathbb{S}_{L+\kappa}(\gamma) 
    := \int_0^T \bigg( L \big( \gamma(t), \dot{\gamma}(t) \big) + \kappa \bigg) \, \mathrm{d}t,
    \quad 
    \gamma \in H^1([0,T], G).
\end{equation}
Analogously to the finite-dimensional case, the dynamics of Tonelli Lagrangians admit a dual formulation. Indeed, the Tonelli assumptions in \Cref{Def: rightv inv Tonelli Langrangian} guarantee that the following map induced by \( L \), called the \emph{Legendre transform},
\begin{equation}\label{eq:defi Legendre transform}
    \mathcal{L} \colon TG \longrightarrow T^*G,
    \quad (x, v) \longmapsto \big(x, (D_v L)_{(x, v)}\big),
\end{equation}
is a diffeomorphism (see~\Cref{lem: Tonelli implies legendre is diffeo} for details).
As in the finite-dimensional case, we can use this to associate to a given Tonelli Lagrangian \( L \) a Hamiltonian by  
\begin{equation}\label{eq: Tonelli Hamiltonian}
    H := E \circ \mathcal{L}^{-1} \colon T^*G \longrightarrow \mathbb{R},
\end{equation}
as its Legendre dual, which is often called a \textit{Tonelli Hamiltonian} in the literature.
The Legendre transform then conjugates the Euler--Lagrange flow of \( L \) to the Hamiltonian flow of \( H \).

Unlike the flow lines of the kinetic Lagrangian, which correspond to geodesics, the flow lines of the Euler–Lagrange flow associated with a general Tonelli Lagrangian cannot, in general, be reparametrized to have unit speed. 
This can be seen, for instance, for electromagnetic Lagrangians \eqref{eq: electromagnetic lagrangians int}, 
since the kinetic term there scales quadratically with speed, whereas the magnetic term $\theta$ scales only linearly with speed. 
Therefore, one of the main points of interest is to understand the similarities and differences between the flow lines of the Euler–Lagrange flow of the kinetic and general Tonelli Lagrangians.  
\medskip

\paragraph{\textbf{Mañé's critical values.}} On finite-dimensional closed manifolds an important role is played by \emph{Mañé’s critical values}~\cite{CIPP, Man}, which marks energy thresholds indicating significant dynamical and geometric transitions in the Euler--Lagrange flow. This notion can be extended to our infinite-dimensional setting. The relevant energy values are the numbers
\[
\min E \le e_0(L) \le c_u(L) \le c_0(L) \leq c(L),
\]
where \( e_0(L) \) is the maximal critical value of \( E \); 
the \emph{lowest Mañé critical value} \( c_u(L) \) is minus the infimum of the mean Lagrangian action
\begin{equation}\label{eq: mean lagrangian action}
    \frac{1}{T} \int_0^T L(\gamma(t), \dot{\gamma}(t))\,\d t,
\end{equation}
over all contractible closed curves \( \gamma \); 
and \( c_0(L) \), called the \emph{strict Mañé critical value}, is minus the infimum of the mean Lagrangian action~\eqref{eq: mean lagrangian action}
over all null-homologous closed curves. 
Finally, the \emph{Mañé critical value} $c(L)$ is defined as minus the infimum 
of the mean Lagrangian action~\eqref{eq: mean lagrangian action} taken over all closed curves. 
For a precise formulation, we refer to \Cref{Def: Manes value for Tonelli Langrangians}. 
We are now in a position to state the main theoretical contribution of this article.

\subsection{Main results} The central goal of this work is to extend the Hopf–Rinow theorem~\cite[Thm.~7.7]{Bauer_2025} to the setting of (strong) Tonelli Lagrangians on half-Lie groups. We begin by stating the analogue of geodesic completeness for the Euler--Lagrange flow of a Tonelli Lagrangian.

\begin{thm}\label{prop: global existince of flow lines}
    Let $L \colon TG \to \mathbb{R}$ be a Tonelli Lagrangian on $(G, \mathcal{G})$. Then every flow line of the Euler–Lagrange flow of $L$ on $TG$ is maximally defined on all of~$\mathbb{R}$. Equivalently, every flow line of the Hamiltonian flow on $T^*G$ associated with the Legendre dual $H$ of $L$ is maximally defined on all of~$\mathbb{R}$. \\
    Moreover, this completeness statement also holds for the restriction of the Lagrangian system 
    \((TG^\ell, L)\) for all $\ell \geq 1$ on the weak Riemannian manifolds \((G^\ell, \mathcal{G})\), 
    where $\mathcal{G}$ denotes the restriction of the Riemannian metric and $L$ the restriction of the 
    Lagrangian from $G$ to $G^\ell$, with \( G^\ell \) denoting the space of \( C^\ell \)-elements in \( G \) in the sense of \Cref{Def: ck elements}.
\end{thm}
In the case of geodesic flows in infinite dimensions, the Hopf--Rinow theorem (and more precisely the statement that metric completeness implies geodesic convexity) fails \cite{HopfrinowfalseAktkin}. 
Motivated by this, we address the existence of global minimizers of the Lagrangian action functional connecting two given points for energy values above Mañé’s critical value, under mild regularity and completeness assumptions on the half-Lie group~$G$.

In order to state our next result, we first fix some notation. 
For two points \( p, q \in G \), we denote by
\begin{equation*}
    \mathcal{P}(p, q) 
    = \{ \gamma \in H^1([0, T], G) \mid \gamma(0) = p, \; \gamma(T) = q, \; T > 0 \}
\end{equation*}
the space of \( H^1 \)-paths connecting \( p \) and \( q \). 
A path \( \gamma \in \mathcal{P}(p, q) \) is called a \textit{global minimizer} of \( \mathbb{S}_{L + \kappa} \) if 
\begin{equation*}
    \mathbb{S}_{L + \kappa}(\gamma) 
    = \inf_{\eta \in \mathcal{P}(p, q)} \mathbb{S}_{L + \kappa}(\eta).
\end{equation*}

\begin{thm}\label{thm:existence of global minimizers}
Let $L : TG \to \RR$ be a Tonelli Lagrangian on $(G,\G)$. 
Assume that \( G \) is \( L^2 \)-regular and that, for all \( p \in G \), the set
\begin{equation*}
    A_{p} := \Bigl\{ \xi \in L^2([0,1], T_e G) \,\Big|\, \operatorname{evol}(\xi) = p \Bigr\}
\end{equation*}
is weakly closed. 
(We refer to \Cref{def:regular half Lie group} for the notions of $L^2$-regularity and the evolution map~$\operatorname{evol}$.)
Then, for every pair of points \( p, q \in G \) and every \( \kappa > c(L) \), there exists a global minimizer of \( \mathbb{S}_{L + \kappa} \) in \( \mathcal{P}(p, q) \).
\end{thm}
\begin{rem}[An infinite-dimensional Tonelli-type theorem]
    \Cref{thm:existence of global minimizers} can be regarded as an infinite-dimensional analogue of the classical Tonelli theorem. 
    In particular, it extends \cite[Thm.~3.1.1]{ContrerasIturriaga1999} from the finite-dimensional to this infinite-dimensional setting.
\end{rem}
\begin{rem}
    \Cref{thm:existence of global minimizers} does not require \( \mathbb{S}_{L+\kappa} \) to be Fréchet differentiable.
\end{rem}

In the case where the action functional \( \mathbb{S}_{L+\kappa} \) is continuously differentiable, by \Cref{lemma:ELE anc solutions} the critical points of \( \mathbb S_{L + \kappa} \) are precisely the flow lines of the Euler--Lagrange flow of \( L \) with prescribed energy \( \kappa \).
Hence, the global minimizers obtained in \Cref{thm:existence of global minimizers} are solutions of the Euler--Lagrange equations with energy \( \kappa \). 
We can therefore conclude that the Euler–Lagrange flow is \emph{convex} for energy values above Mañé’s critical value, in the sense that any two points in \( G \) can be connected by a solution of the Euler–Lagrange equations with prescribed energy. This in particular applies to strong Tonelli Lagrangians, as the strong Tonelli property implies that \( \mathbb S_{L + \kappa} \) is \( C^1 \). We summarize this observation in the following theorem.

\begin{thm}\label{thm: Hopf-Rinow for Tonelli int} 
Let $L : TG \to \RR$ be a strong Tonelli Lagrangian on $(G,\G)$.
Under the same assumptions of \Cref{thm:existence of global minimizers}, for all energy levels $\kappa > c(L)$ and all $p,q \in G$ there exists a solution of the Euler--Lagrange equations of $L$ with energy $\kappa$ that connects \( p \) and \( q \) and minimizes the action $\mathbb{S}_{L+\kappa}$ over \( \mathcal P(p, q) \).

\end{thm}
\begin{rem}[From finite to infinite dimensions]
    To the best of the author's knowledge, \Cref{thm: Hopf-Rinow for Tonelli int} 
    is the first result of this kind for Tonelli Lagrangians on 
    infinite-dimensional manifolds. 
    Also, one can think of \Cref{thm: Hopf-Rinow for Tonelli int} as an extension of the results in~\cite[Cor.~B]{Co06} 
    from closed finite-dimensional manifolds to the infinite-dimensional setting.
\end{rem}

\begin{rem}[Kinetic case]
    The kinetic Lagrangian (cf.\ \cref{rem: intro magnetic Langrang are Tonelli int} with \( \theta = 0\) and \( V = 0\)) is Tonelli. Thus, \Cref{thm: Hopf-Rinow for Tonelli int} extends statements~(d) and~(e) of \cite[Thm.~7.7]{Bauer_2025} to our setting.
\end{rem}
\begin{rem}
    If \( \kappa \le c(L) \), there may exist pairs of points that cannot be connected by 
    a flow line of the Euler--Lagrange flow with energy \( \kappa \); see~\cite{M24}. 
    This phenomenon already occurs in the finite-dimensional case, for instance on 
    \( S^3 = SU(2) \); see ~\cite{ALBERS2025105521}. 
    However, even for finite-dimensional closed manifolds, there is in general no 
    satisfactory understanding of which pairs of points fail to be connected for a 
    given subcritical energy.
\end{rem}
Before moving on, we point out that the proof of \Cref{thm: Hopf-Rinow for Tonelli int} carries over, with minor modifications, to all Tonelli Lagrangians for which \( \mathbb{S}_{L+\kappa} \) is \( C^1 \). 
It would be interesting to determine whether there exist classes of Tonelli Lagrangians beyond the strongly Tonelli ones for which this property still holds. 

We close this subsection by mentioning that, if the universal cover \( \pi \colon \hat{G} \to G \) satisfies an analogous condition to that of \( G \) in \Cref{thm: Hopf-Rinow for Tonelli int}, 
then the energy threshold appearing therein can be lowered to the lowest Mañé critical value. 

\begin{cor}\label{cor: hopf rinow above lowest manes critcal value}
    Let $L : TG \to \RR$ be a strong Tonelli Lagrangian on $(G, \G)$. 
    Suppose that for its universal cover $\pi: \hat G \rightarrow G$, equipped with the pullback metric $\hat \G$, 
    it holds that $\hat G$ is $L^2$–regular and that for each $\hat p \in \hat G$ the sets
    \[
    A_{\hat p} := \Bigl\{ \xi \in L^2([0,1], T_e \hat G) \;\big|\; \mathrm{evol}(\xi) = \hat p \Bigr\} 
    \subset L^2([0,1], T_e \hat G)
    \]
    are weakly closed. 
    Then, for all energy levels $\kappa > c_u(L)$ and all $p, q \in G$, 
   there exists a solution of the Euler--Lagrange equations of $L$ with energy $\kappa$ that connects \( p \) and \( q \) and minimizes the action $\mathbb{S}_{L+\kappa}$ over \( \mathcal P(p, q) \).
\end{cor}

\begin{proof}
The proof goes as follows. 
First, one observes that the lifted Tonelli Lagrangian $\hat{L} := L \circ \pi$ is an $\hat{G}$–right invariant Tonelli Lagrangian on $(\hat{G}, \hat{\G})$. 
This follows from the fact that $\pi$ is a topological group homomorphism and a local isometry, 
and that uniform convexity is a local condition. 
Next, we observe that $\pi$ maps Euler--Lagrange flow lines of $\hat{L}$ with energy $\kappa$ to flow lines of $L$ with the same energy $\kappa$. 
By the definition of $c_u(\hat{L})$ and $c(\hat{L})$ in \eqref{eq: mean lagrangian action} or \Cref{Def: rightv inv Tonelli Langrangian}, on simply connected spaces we have the equality $c(\hat{L}) = c_u(\hat{L})$. 
This discussion, in combination with the application of \Cref{thm: Hopf-Rinow for Tonelli int} 
to $\hat{L}$ on $(\hat{G}, \hat{\G})$, finishes the proof.
\end{proof}

\begin{rem}
   \Cref{cor: hopf rinow above lowest manes critcal value} is an extension of the results in~\cite[Cor.~B]{Co06} 
    from closed finite-dimensional manifolds to the infinite-dimensional setting.
\end{rem}
\begin{rem}[Relation to \cite{HopfRinowHalfLiegroups}]
  In the case of a magnetic Lagrangian, i.e.\ a Lagrangian of the form described in 
\Cref{rem: intro magnetic Langrang are Tonelli int} with \( V = 0 \), this result 
generalizes the recent results (4) and (5) of the authors in 
\cite[Thm.~1.6]{HopfRinowHalfLiegroups}, since \( c_u \) is in general less than or equal to 
the geometric version of Mañé’s critical value defined therein. \\ 
In the case of a non-exact magnetic system, there may or may not be a relation between 
\Cref{cor: hopf rinow above lowest manes critcal value} and 
\cite[Thm.~1.6]{HopfRinowHalfLiegroups}. 
\end{rem}

\subsection{First illustrations of \Cref{thm: Hopf-Rinow for Tonelli int}} In the following, we illustrate \Cref{thm: Hopf-Rinow for Tonelli int} 
first through a purely geometric application to groups of Sobolev diffeomorphisms, 
and second through an application to geometric hydrodynamics via the so-called 
\emph{magnetic Euler–Arnold equations} and \emph{Euler–Poincaré equations}.\medskip

\textbf{Euler–Lagrange flow of Tonelli Lagrangians on Sobolev diffeomorphism groups.}
Let \( (M, g) \) be a compact, finite-dimensional Riemannian manifold. 
We consider the group of Sobolev diffeomorphisms \( \mathrm{Diff}^{H^s}(M) \) 
of Sobolev order \( s > \frac{\dim M}{2} + 1 \). For more details, we refer to \Cref{ex: sobolov lie group}. 
We equip this infinite-dimensional half-Lie group with the strong, right-invariant Sobolev metric of order~\( s \), defined by
\begin{equation*}
    \mathcal{G}^s_{\varphi}(h \circ \varphi,\, k \circ \varphi) 
    = \int_M g\big( (1 - \Delta)^{s/2} h,\; (1 - \Delta)^{s/2} k \big)\, \mathrm{dvol}_g, 
    \qquad \forall\, \varphi \in \mathrm{Diff}^{H^s}(M),
\end{equation*}
where \( h \) and \( k \) are \( H^s \)-vector fields on \( M \), and where 
\( \Delta \) denotes the Laplacian with respect to the Riemannian volume form \( \mathrm{dvol}_g \) associated with the Riemannian metric \( g \). 

Using the results of~\cite{EM70,BauerBruverisCismasEscherKolev2020,BauerEscherKolev2015}, 
one sees that \( \mathcal{G}^s \) is indeed a strong Riemannian metric on \( \mathrm{Diff}^{H^s}(M) \). 
For non-integer values of \( s \), this relies on highly non-trivial analytic estimates, 
which have been established only recently~\cite{BauerBruverisCismasEscherKolev2020,BauerEscherKolev2015}. 

Thus, we conclude from \Cref{prop: global existince of flow lines} that for every Tonelli Lagrangian $L$ on $ \big(\mathrm{Diff}^{H^s}(M), \mathcal{G}^s \big)$, 
its associated Euler–Lagrange flow (as well as the associated dual Hamiltonian flow) is maximally defined on all of $\RR$. Note that in order to deduce this one first has to verify the assumptions in \Cref{thm:existence of global minimizers} on the base \( \mathrm{Diff}^{H^s}(M) \).
These were implicitly shown in~\cite{BruverisVialard2017} for \( \mathrm{Diff}^{H^s}(M) \) as well as for its universal cover.
Thus, for all strong Tonelli Lagrangians on
\( (\mathrm{Diff}^{H^s}(M), \mathcal{G}^s)\), 
convexity of the Euler-Langrange flow holds for energies above the lowest Mañé’s critical value $c_u(L)$.  

By choosing the kinetic Lagrangian \( L = \tfrac{1}{2} \|v \|_{\mathcal{G}}^2 \), whose associated Euler–Lagrange flow is precisely the geodesic flow, 
our result reduce to the completeness results for the group of Sobolev diffeomorphisms 
as obtained in~\cite{BruverisVialard2017}. 
\medskip

\textbf{Euler–Poincaré equations and magnetic Euler–Arnold equations.} 
We close this subsection with an application of \Cref{thm: Hopf-Rinow for Tonelli int} 
to geometric hydrodynamics, in the context of the \emph{Euler–Poincaré equations} introduced by Holm–Marsden–Ratiu in \cite{HolmMardsenRatiuEulerPoincare}, 
or more specifically in the context of the \emph{magnetic Euler–Arnold equation} 
recently introduced by the first author in~\cite{maier2025geometrichydrodynamicsinfinitedimensional}. 
We know that for the restriction of the Lagrangian system 
considered above to the group of Sobolev diffeomorphisms, restricted to its \( C^k \) 
and, in particular, \( C^1 \) elements, the completeness statement for the Euler--Lagrange flow in \Cref{prop: global existince of flow lines} holds. 
That is, the Euler–Lagrange flow of the system consisting of a Tonelli Lagrangian on 
\begin{equation}\label{eq: magn sysetm on sobolev lie groups}
\big( \mathrm{Diff}^{H^{s+1}}(M),\, \mathcal{G}^s\big)
\end{equation}
exists globally, where we have used that the \( C^1 \)-elements in \( \mathrm{Diff}^{H^{s}}(M) \) are precisely \( \mathrm{Diff}^{H^{s+1}}(M) \).

Using \cite[Lemma 7.4]{Bauer_2025}, the adjoint $\mathrm{ad}^T$ of \( \mathrm{ad} \) exists, thus by the following well-known identity (see for example \cite{AK98,Vi08} or \cite[eq. (2.8)]{maier2025geometrichydrodynamicsinfinitedimensional}) \[
\operatorname{ad}^*_u (A v) = -A \big( \operatorname{ad}^\top_u (v) \big) \quad \forall\, u, v \in T_e \mathrm{Diff}^{H^{s+1}}(M).
\]
$\operatorname{ad}^*$ also exists. Here, $A$ is given on $(\mathrm{Diff}^{H^{s}}(M), \G^s)$ by \( A_s = (1 - \Delta_g)^s \) and it is usually referred to in the literature as the \textit{inertia operator}. By using~\cite[Thm. 1.2]{HolmMardsenRatiuEulerPoincare} the curve \( \varphi \) is a flow line of the Euler--Lagrange flow of $L$ on the space in ~\eqref{eq: magn sysetm on sobolev lie groups} if and only if \( u:= \dot{\varphi}\circ\varphi^{-1} \) is a solution of the basic Euler--Poincaré equations:
\begin{equation}\label{eq: EP eq}
    \frac{\d}{\d t}\frac{\delta \ell}{\delta u} = \operatorname{ad}^*_{\delta u}\frac{\delta \ell}{\delta u} \tag{EP},
\end{equation}
where $\ell(\cdot)=L(e,\cdot)$ and $\frac{\delta \ell}{\delta u}\in T^*_e \mathrm{Diff}^{H^{s+1}}(M)$ denotes the functional derivative. For a precise definition we refer to \cite{HolmMardsenRatiuEulerPoincare}.  By using this duality in combination with \Cref{prop: global existince of flow lines} we obtain global well-posedness for \eqref{eq: EP eq} for any Tonelli Lagrangian on  the space~\eqref{eq: magn sysetm on sobolev lie groups}.\\

This can be made more precise for right-invariant electromagnetic Lagrangians on the space.
That is, by using the results of \cite{Khesin-Mis-Mod-inf-Newton} in combination with \cite[Cor.~2.9, Thm.~2.10]{maier2025geometrichydrodynamicsinfinitedimensional}, 
we obtain that a curve \( \varphi \) is a flow line of the Euler--Lagrange flow of a right-invariant electromagnetic Lagrangian (cf.\ \eqref{eq: electromagnetic lagrangians int}) if and only if \( u := \dot{\varphi}\circ\varphi^{-1} \) is a solution of the following partial differential equation:
\begin{equation}\label{eq: MEpDiff}
    m_t + \nabla_u m + (\nabla u)^{\!\top} m + (\operatorname{div} u)\, m 
    = -\, A_s\big(Y_{\id}(u)\big) - A_s(\nabla V(u)), 
    \tag{ELMEpDiff}
\end{equation}
where \( Y \) is the Lorentz force defined by 
\[
\G^s_x(Y_x(u), v) = \d\theta_x(u, v) \quad \forall x \in G, \ \forall u, v \in T_x G,
\]
\( m = A_s u \) is the \emph{momentum density} 
and \( \nabla \) is the Levi–Civita connection of \( (M, g) \).\\

We call~\eqref{eq: MEpDiff} the \emph{electromagnetic EPDiff equation}, 
which in the case of a vanishing potential is precisely a special case of the \emph{magnetic EPDiff equation} introduced by the authors in \cite{HopfRinowHalfLiegroups}. 
In the case of a vanishing magnetic field \( \d\theta = 0 \) and vanishing potential \( V = 0 \), 
we recover the classical EPDiff equation, the geodesic equation of \( \big( \mathrm{Diff}^{H^s}(M), \G^s \big) \).

As these electromagnetic Lagrangians are all Tonelli by the discussion above,  
we obtain global well-posedness for~\eqref{eq: MEpDiff} on the space of vector fields \( u \) of Sobolev class $s$ at least \(\tfrac{\dim M}{2} + 1 \).
\subsection{Outline of the paper.}
In \Cref{s:preliminaries}, we recall the notion of half-Lie groups and their \( C^k \)-elements, 
as well as strong Riemannian metrics.
In \Cref{s:Tonelli Lagrangians on half Lie-Groups}, we provide all details outlined in the previous paragraph 
on right-invariant Tonelli Lagrangians on half-Lie groups. 
In particular, we present background on the time-free action functional. 
Moreover, we show that the no-loss–no-gain result of~\cite{Bauer_2025} 
extends to flow lines of the Euler–Lagrange flow of a Tonelli Lagrangian, that is the Euler--Lagrange flow does not alter regularity.
In \Cref{s: Mañé's critical values for Tonelli Lagrangians.}, we provide details on the Mañé critical values.
Finally, in \Cref{s: The Hopf–Rinow theorem for Tonelli Lagrangians.}, 
we present the proofs of \Cref{prop: global existince of flow lines},\Cref{thm:existence of global minimizers}, and \Cref{thm: Hopf-Rinow for Tonelli int}.
\\

\paragraph{\textbf{Acknowledgments.}} The authors thank P. Albers and all participants of the symplectic research seminar in Heidelberg for their valuable feedback on earlier versions of this article.
L.M.\ thanks A.~Abbondandolo, G.~Benedetti, B.~Khesin, and L.~Assele for many helpful discussions on Hamiltonian systems.
L.M.\ would also like to thank M.~Bauer and P.~Michor for insightful discussions about their work~\cite{Bauer_2025}, as well as the participants of \emph{Math en plein air 2025} and \emph{Dynamische Systeme MFO 2025} for valuable feedback on earlier versions of this work.\\
The authors acknowledge funding by the Deutsche Forschungsgemeinschaft (DFG, German Research Foundation) – 281869850 (RTG 2229), 390900948 (EXC-2181/1), and 281071066 (TRR 191). L.M.\ gratefully acknowledges the excellent working conditions and stimulating interactions at the Erwin Schrödinger International Institute for Mathematics and Physics in Vienna during the thematic programme \emph{``Infinite-dimensional Geometry: Theory and Applications''}, where part of this work was carried out. \\
Finally, L.M. thanks F. Schlenk for his warm hospitality at the University of Neuchâtel in October 2025, during which part of this work was completed.
\section{Preliminaries}\label{s:preliminaries}
We begin by fixing some notation. For a group $G$ we denote by $\mu \colon G \times G \to G$ the group multiplication and by $\mu_x, \mu^y$ left and right translations respectively:
\[
\mu(x, y) = \mu_x(y) = \mu^y(x).
\]

\subsection{Half-Lie groups}\label{ss: Half Lie grous}
Let us give a precise definition of half-Lie group and discuss a few examples. We follow the presentation in \cite{Bauer_2025}.
\begin{defn}
    A \textit{right (resp.\ left) half-Lie group} is a topological group endowed with a smooth structure such that right (resp.\ left) multiplication by any element is smooth.
\end{defn}
\begin{rem}
    Half-Lie groups can be modeled on various different types of spaces. In this work we stick to Banach (and Hilbert) half-Lie groups.
\end{rem} 
The main examples include groups of $H^s$- or $C^k$-diffeomorphisms (see~\cite{EM70}), 
as well as semidirect products of a Lie group with kernel an infinite-dimensional representation space. 
Since the first case will serve as our guiding example, we now describe it in more detail.

\begin{example}\label{ex: sobolov lie group}
If $(M, g)$ is a finite-dimensional compact Riemannian manifold or an open Riemannian manifold of bounded geometry, 
then the diffeomorphism group $\mathrm{Diff}^{H^s}(M)$ of Sobolev regularity $s > \dim(M)/2 + 1$ is a half-Lie group. 
Likewise, the groups  $\mathrm{Diff}^{C^k}(M)$ for $1 \leq k < \infty$ 
are half-Lie groups. However, they are not Lie groups because left multiplication is non-smooth.
For a thorough explanation, we refer to~\cite{EM70}.
\end{example}

While left multiplication is only required to be continuous, some elements display better regularity. This is captured by the following definition.
\begin{defn}\label{Def: ck elements}
    An element \( x \in G \) of a Banach half-Lie group is of class \( C^k \) if the left translations \( \mu_x, \mu_x^{-1} \colon G \to G \) are \( k \) times continuously Fréchet differentiable. We denote by \( G^k \) the set of \( C^k \) elements.
\end{defn}

\begin{example}
\label{Ex:Diff_differentiable_elements}
Let $M$ be a closed manifold. Then by~\cite[Ex. 3.5]{Bauer_2025} it holds that
\[
\big(\mathrm{Diff}^{C^k}(M)\big)^{(\ell)} = \mathrm{Diff}^{C^{k+\ell}}(M).
\]
\end{example}
For further details on $C^k$-elements of half-Lie groups, we refer the reader to \cite[§§2–3]{Bauer_2025}. 

We now recall the notion of regular half-Lie groups in the setting where the half-Lie group carries a Banach manifold structure.
\begin{defn}[Regular half-Lie groups]\label{def:regular half Lie group}
Let $G$ be a Banach half-Lie group, and let $F$ be a subset of  $L^1_{\mathrm{loc}}(\RR, T_e G)$.  
We say that $G$ is \emph{$F$-regular} if for every $X \in F$ there exists a unique solution $ g \in W^{1,1}_{\mathrm{loc}}(\RR, G)$ of the differential equation
\[
    \partial_t g(t) \;=\; T_e \mu_{g(t)}\, X(t), 
    \qquad g(0) = e.
\]
This solution is called the \emph{evolution} of $X$ and is denoted by $\mathrm{Evol}(X)$. 
Its evaluation at $t=1$ is denoted by $\mathrm{evol}(X)$.
\end{defn}
\subsection{Riemannian metrics on half-Lie groups}
Following \cite{Bauer_2025}, we recall the notion of strong right-invariant Riemannian metrics.
\begin{defn}\label{Def: strong right invariant Riemannian metrics}
A Riemannian metric \( \mathcal{G} \) on a half-Lie group \( G \) is called 
\emph{$G$-right-invariant} if
\[
  \mathcal{G}_x\big( T_e \mu^x v,\, T_e \mu^x w \big)
  = \mathcal{G}_e(v, w),
  \quad \forall\, x \in G,\, v, w \in T_e G .
\]
The Riemannian metric \( \mathcal{G} \) is called \emph{strong} if 
\( \mathcal{G}_x \) induces the manifold topology on \( T_x G \) for every \( x \in G \); it is called \emph{weak} otherwise.
\end{defn}

Kriegl and Michor \cite{KrieglMichor1997} define \textit{convenient manifolds} as manifolds modeled on convenient vector spaces, i.e. Mackey complete locally convex spaces. All Banach and Fréchet spaces are convenient. Strong Riemannian metrics on such manifolds admit the following useful characterization.
\begin{thm}[{\cite[Theorem 7.2]{Bauer_2025}}]
    Let \( M\) be a convenient manifold equipped with a Riemannian metric $\G$. Then, \( \G \) is a strong Riemannian metric if and only if $M$ is a Hilbert manifold and the canonical map
    \begin{equation*}
\mathcal{G}^\vee \colon TM \longrightarrow T^*M, 
\quad (x, v) \mapsto \mathcal{G}_x(v, \cdot) \,  ,
\end{equation*}
    is a bundle isomorphism.
\end{thm}

\section{Tonelli Lagrangians on half Lie-Groups}\label{s:Tonelli Lagrangians on half Lie-Groups}
This section is devoted to the precise definition of \emph{Tonelli Lagrangians}, 
their relation---via the \emph{Legendre transform}---to the corresponding 
\emph{Tonelli Hamiltonians}, and the fact that their flows are conjugated 
by the Legendre transform. 
Finally, we describe the flow lines of the \emph{Euler--Lagrange flow} 
of a Tonelli Lagrangian with prescribed energy as the 
\emph{critical points of the free-period action functional}.
\subsection{Definition of Tonelli Lagrangians on half Lie-Groups}\label{ss:Definition of Tonelli Lagrangians on half Lie-Groups}
Let us define our object of interest, namely Tonelli Lagrangians. Because of the non-compact nature of infinite-dimensional Banach manifolds, we have to make certain assumptions on the growth of Tonelli Lagrangians, similarly to what is done in \cite{AbbondandoloSchwarz2009}.
\begin{defn}[Tonelli Lagrangians on half-Lie groups]\label{Def: rightv inv Tonelli Langrangian}
  Let $(G, \mathcal{G})$ be a half-Lie group equipped with a strong right-invariant Riemannian metric~$\mathcal{G}$. 
We call a $G$-right-invariant $C^2$ function $L \colon TG \to \mathbb{R}$ 

\begin{enumerate}
    \item a \emph{Tonelli Lagrangian} if there exist constants $M \geq m > 0$ such that
    \begin{equation}\label{eq:uniform convexity}
        m \operatorname{Id}_{T_eG} \leq \operatorname{Hess}_v L(e,v) \leq M \operatorname{Id}_{T_e G}
    \end{equation}
    for all \( v \in T_eG \);
    \item a \emph{strong Tonelli Lagrangian} if it is Tonelli and, in addition, there exists a constant $C > 0$ such that
    \begin{equation}\label{eq: tonelli differential in x}
        \| D_x L(x,v) \| \leq C \, (1+\| v \|_{\mathcal{G}}^2)
    \end{equation}
    for all $x$ in a neighborhood of the identity of $G$.
\end{enumerate}
\end{defn}
\begin{rem}
    The definition of Tonelli Lagrangian depends on the metric $\G$ on $G$. In our presentation \( \mathcal{G} \) is chosen at the outset, so we will not make further reference to this.
\end{rem}
Let us briefly comment on the notion of Hessian and how the above definition should be understood. By \cite[Thm.~7.1]{Bauer_2025}, the tangent space \( T_e G \) of $G$ at the identity is a Hilbert space. 
Consider now the function
\begin{equation*}
    L(e, \cdot) \colon T_e G \to \mathbb{R}, \quad v \mapsto L(e, v).
\end{equation*}
Its differential is a map 
\[
D_v L(e, \cdot) \colon T_e G \to T_e^* G \, 
\] 
and differentiating once more yields a map 

\[
 D^2_v L(e, \cdot) \colon T_e G \to L(T_e G, T_e^* G) \cong L^2(T_e G, \mathbb{R}) \, ,
\]
where the latter is the space of continuous bilinear maps \( T_e G \times T_e G \to \mathbb{R} \). 
Moreover, \( D^2_v L(e, \cdot) \) is symmetric. 
By the Riesz representation theorem, for all \( v \in T_e G \) there exists a unique self-adjoint operator 
\[
 A_v \colon T_e G \to T_e G
\]
satisfying
\begin{equation*}
    \mathcal{G}_e( A_v u, w ) = D^2_v L(e, v)(u, w) \quad \forall\, u, w \in T_e G.
\end{equation*}
The Hessian of \( L(e, \cdot) \) at \( v \) is by definition \( A_v \), 
and we will denote it by \( \operatorname{Hess}_v L (e, v) \), 
where the subscript stands for \enquote{vertical} to signify that we are differentiating only along the fibers of \( TG \to G \). 
Thus~\eqref{eq:uniform convexity} is to be understood as
\[
 m \operatorname{Id}_{T_eG} \leq A_v \leq M \operatorname{Id}_{T_{e}G},
\]
which by definition means
\[
 m \| w \|^2_{\mathcal G} \leq \G_e(A_v w,\, w) \leq M \| w \|_{\G}^2
 \qquad \forall\, w \in T_e G.
\]
The requirement \( m \operatorname{Id}_{T_eG} \leq \operatorname{Hess}_v L (e, \cdot) \) can be thought of a uniform convexity condition on the function \( v \mapsto L(e, v) \).
\begin{rem}
Compared to the definition of a Tonelli Lagrangian on a finite-dimensional, possibly non-compact manifold, given in~\cite[§1.1]{ContrerasIturriaga1999}, we do not need to impose the conditions~(b) of superlinearity or~(c) of boundedness stated therein, since for the Tonelli Lagrangians in \Cref{Def: rightv inv Tonelli Langrangian} these properties are automatically satisfied.
 This is because the Lagrangians considered in \Cref{Def: rightv inv Tonelli Langrangian} are $G$-right-invariant, and we will see that in this situation uniform convexity already implies the other two conditions. 
\end{rem} 

Similarly to the finite-dimensional case, a Tonelli Lagrangian exhibits quadratic growth on each fiber. We formalize this observation in the following lemma.
\begin{lem}\label{Ilem: Tonelli Langrangians}
    Let $L:TG \to \RR$ be a Tonelli Lagrangian on a half-Lie group $G$. Then \( L \) grows  quadratically, i.e.\ there exist constants \( M, m, b > 0 \) such that
        \begin{equation}\label{eq: quadratic growth of Tonelli langrang on half lie groups}
          \tfrac{M}{4} \Vert v\Vert_{\G}^2 + b\geq  L(x,v) \geq \tfrac{m}{4} \Vert v\Vert_{\G}^2 - b \quad \forall\, (x,v)\in TG
        \end{equation}
        for some constant \( b \).
        In particular, $L$ exhibits superlinear growth on each fiber, i.e.\ 
        \[
            \lim_{\Vert v\Vert_{\G}\to \infty} \frac{L(x,v)}{\Vert v\Vert_{\G}} = +\infty 
            \quad \forall\, x \in G\, .
        \]
\end{lem}

\begin{proof}
    By right-invariance of \( L \) and \( \mathcal{G} \), it is sufficient to prove the lemma at the identity \( e \in G \). Let us write \( f = L(e, \cdot) \) for the sake of brevity. Let \( v \in T_e G \) and consider the function \( \varphi(t) = f(tv), \; t \in [0, 1] \). A straightforward computation shows that
    \begin{equation*}
        \varphi''(t) = \mathcal{G}_e \big(v, \operatorname{Hess}f (e) v \big) \geq m \| v \|^2_\mathcal{G},
    \end{equation*}
    which in turn implies that \( \varphi'(t) \geq \varphi'(0) + t m \| v \|^2_\mathcal{G} \).
    Thus,
    \begin{equation}\label{eq:bound}
        f(v) - f(0) = \int_0^1 \varphi'(t) \mathop{}\!\mathrm{d}t
        \geq \varphi'(0) + \frac{m}{2} \| v \|^2_\mathcal{G}.
    \end{equation}
    The only thing that remains to do is bound \( \varphi'(0) \). This is done as follows:
    \begin{equation*}
       \varphi'(0) = \d_0 f(v) 
            \geq - \| \nabla f(0) \|_\mathcal{G} \| v \|_\mathcal{G} 
            \geq - \frac{1}{m} \| \nabla f(0) \|_\mathcal{G}^2 - \frac{m}{4} \| v \|_\mathcal{G}^2.
    \end{equation*}
    Plugging this into \eqref{eq:bound} completes the proof of the right-hand inequality in \eqref{eq: quadratic growth of Tonelli langrang on half lie groups}. The left-hand inequality in~\eqref{eq: quadratic growth of Tonelli langrang on half lie groups} follows from a similar argument. This concludes the proof.
\end{proof}

\begin{rem}\label{rem: fibrewise differential grows quadratically}
    By using the same methods as in the above proof, one can similarly conclude that
\begin{equation*}
    \| D_v L (e, v) \| \leq M \bigl(1 + \| v \|_\mathcal{G}\bigr)
\end{equation*}
for all \( v \in T_e G \).
\end{rem}
We close this subsection by noting that, on closed finite-dimensional manifolds, Tonelli Lagrangians have been studied extensively, as they provide a natural framework for the application of variational methods. Moreover, for Tonelli Lagrangians, the induced Legendre transform is a diffeomorphism, allowing one to move seamlessly between the Hamiltonian and Lagrangian formulations—a fact that we make precise in the next paragraph in our infinite-dimensional setting.\medskip

\subsection{The Legendre transform.}
Next, from the right invariance of $L$ together with the uniform convexity of $L$ (see~\eqref{eq:uniform convexity}), 
one concludes that the Legendre transform is not only well defined but in fact a diffeomorphism.

\begin{lem}\label{lem: Tonelli implies legendre is diffeo}
    Given a Tonelli Lagrangian $L: TG \to \RR$ on $(G, \G)$, 
    the Legendre transform associated to $L$, that is,
    \begin{equation}\label{eq: defi legendre transform}
          \mathcal{L} : TG \;\to\; T^*G, 
    \quad (x,v) \;\mapsto\; \bigl(x, D_v L(x,v)\bigr),
    \end{equation}
    is a diffeomorphism.
\end{lem}

\begin{proof}
    Uniform convexity implies that \( \mathcal L \) is a fibrewise isomorphism. This immediately implies that it is a diffeomorphism.
\end{proof}

Thus, by \Cref{lem: Tonelli implies legendre is diffeo}, if $L$ is Tonelli, then it is a hyperregular Lagrangian in the sense of Marsden–Ratiu~\cite[§7.3]{MardsenRatiuMechanics}. 
As already introduced in the introduction in \eqref{eq: def tonelli energy}, and following~\cite[§7.3]{MardsenRatiuMechanics}, we call  
\begin{equation}\label{eq: def tonelli energy}
    E(x,v) := D_v L(x,v)(v) - L(x,v)
\end{equation}
the \emph{energy} of $L$.

\begin{example}
    In the case of a $G$-right invariant electromagnetic Lagrangian (see \Cref{rem: intro magnetic Langrang are Tonelli int}), the energy takes the form  
    \begin{equation}\label{eq: energy electromagnetic}
        E(x,v) = \tfrac{1}{2}\Vert v\Vert_{\G}^2 + V(x) \, .
    \end{equation}
\end{example}

Since the canonical symplectic form $\Omega$ on $T^*G$ is strong and $\mathcal{L}$ is a diffeomorphism, 
the pullback $\mathcal{L}^*\Omega$ is a strong symplectic form on $TG$ (here, strongness of a symplectic form means that induces a bundle isomorphism between the tangent and cotangent bundles).
Thus, the Hamiltonian vector field $X_E$ of the energy $E$ with respect to $\mathcal{L}^*\Omega$ is well defined. 
Its flow is denoted by 
\[
\Phi_L^t : TG \longrightarrow TG, \quad t \in I,
\]
and is called the \emph{Euler--Lagrange flow} of $L$.  
By~\cite[Thm.~7.3.3]{MardsenRatiuMechanics}, the flow lines of the Euler--Lagrange flow \( \Phi_L \) of \( L \) 
are precisely the solutions of the Euler--Lagrange equations associated with \( L \). 
In summary, we have 
\begin{lem}\label{lemm: ele define EL-flow}
    Given a Tonelli Lagrangian $L : TG \to \RR$ on $(G, \G)$, 
    then $(\gamma, \dot{\gamma})$ is a flow line of the Euler–Lagrange flow $\Phi_L$
    if and only if $(\gamma, \dot{\gamma})$ is a solution of the Euler–Lagrange equations \eqref{eq: ELE} of $L$.
    In particular, this implies that the Euler–Lagrange flow $\varPhi_L^t$ preserves the energy $E$.
\end{lem}
As in finite dimensions, the energy $E$ of $L$ as in \eqref{eq: def tonelli energy} gives rise, via the Legendre transform associated to $L$, to the Hamiltonian of a Tonelli Lagrangian, defined by  
\begin{equation}\label{eq: def tonelli Hamiltonian}
    H := E \circ \mathcal{L}^{-1} : T^*G \longrightarrow \RR \, ,
\end{equation}
which is the Legendre dual of $E$. All such Hamiltonians are called Tonelli Hamiltonians.

\begin{rem}
    Equivalently, one could define \emph{Tonelli Hamiltonians} on $(G, \G)$ by requiring $G$–right invariance and similar growth conditions as \eqref{eq:uniform convexity}.
    Tonelli Lagrangians $L$ on $TG$ induce Tonelli Hamiltonians $H$ on $T^*G$ and vice versa. 
\end{rem}

Next, as the canonical symplectic form $\Omega$ on $T^*G$ is strong, the Hamiltonian vector field $X_H$ exists and is uniquely determined. 
By the Picard–Lindelöf theorem, the differential equation 
\[
\dot{\Gamma} = X_H(\Gamma)
\]
is locally well posed, and thus the Hamiltonian flow $\Phi_H$ is locally defined. 
By~\cite[Prop.~7.4.1]{MardsenRatiuMechanics}, the vector fields $X_H$ and $X_E$ are $\mathcal{L}$-related, 
which follows from the fact that $\mathcal{L}$ is, by construction, a symplectomorphism. 
Hence, by~\cite[Thm.~7.4.3]{MardsenRatiuMechanics}, 
the Hamiltonian flow of \( H \) is conjugated via the Legendre transform 
to the Euler–Lagrange flow of \( L \). 

\begin{lem}\label{lem: legendre transform conjugates ELE flow and Hamiltonian flow}
Let \( L \colon TG \to \RR \) be a Tonelli Lagrangian on \( (G, \G) \). 
Then its Euler–Lagrange flow \( \Phi_L \) 
is conjugated, via the Legendre transform \( \mathcal{L} \) as in~\eqref{eq: defi legendre transform} 
induced by the Tonelli Lagrangian \( L \), 
to the Hamiltonian flow of \( H = E \circ \mathcal L^{-1} \) on \( T^*G \).
\end{lem}

We have seen that since $X_H$ and $X_E$ are vector fields of class at least $C^1$ on a Hilbert manifold, their flows exist at least locally by the Picard–Lindelöf theorem. 
However, due to the lack of compactness of half-Lie groups, it is a priori unclear whether these flows exist for all times.

In the next section we will see that, as in the finite-dimensional case, solutions of the Euler–Lagrange equation are critical points of the free-period action functional. 
\subsection{The time-free action functional.} In this section we discuss the time-free Lagrangian action functional, essentially adapting \cite{AbbondandoloSchwarz2009} to our setting.

\medskip

Let \( L \) be a Tonelli Lagrangian on $(G,\G)$ in the sense of \Cref{Def: rightv inv Tonelli Langrangian} and consider the space of all \( H ^ 1 \) curves \( \gamma \colon [0, T] \to G \) for all possible choices of \( T > 0 \). We identify each \( \gamma \) with the pair \( (x, T) \), where
\[
    x \in H^1([0,1],G), 
    \quad 
    x(s) := \gamma(sT).
\]  
For any real number $\kappa$, the time-free action functional corresponding to the energy $\kappa$ is defined as  
\begin{equation}\label{eq: int action funct x,T}
    \SS_{L+\kappa}(x,T) \, 
    := T \int_0^1 \left( L\!\left(x(s), \tfrac{x'(s)}{T}\right) + \kappa \right) \d s
    = \int_0^T \bigg( L \big( \gamma(t), \gamma'(t) \big) + \kappa \bigg) \d t.
\end{equation}
The fact that \( L \) is Tonelli, and more specifically the quadratic growth on fibers, implies that \( \mathbb{S}_{L+\kappa} \) as in \eqref{eq: int action funct x,T} is well-defined on the Hilbert manifold \( H^1([0, 1], G) \times (0, \infty) \).
Therefore, we obtain a functional  
\[
    \SS_{L+\kappa} \colon H^1([0, 1],G) \times (0,\infty) \to \RR.
\]
For a Tonelli Lagrangian, this functional is not necessarily Fréchet differentiable. However, by strengthening the assumptions on \( L \), for instance by requiring that \( L \) be a strong Tonelli Lagrangian one can show the following result, whose proof goes along the same lines as \cite[Prop.~3.1~(i)]{AbbondandoloSchwarz2009}.
\begin{lem}\label{lem: strong tonelli implies action is c1}
    Let \( L : TG \to \mathbb{R} \) be a strong Tonelli Lagrangian on \( (G, \mathcal{G}) \) (cf.\ \Cref{Def: rightv inv Tonelli Langrangian}). Then its associated time-free action functional \( \mathbb{S}_{L+\kappa} \), as defined in \eqref{eq: int action funct x,T}, is of class \( C^1 \) on \( H^1([0,1], G) \times (0,\infty) \).
\end{lem}

\begin{rem}
    We point out that the proof of \Cref{lem: strong tonelli implies action is c1} relies crucially on the assumption that \( L \) is a strong Tonelli Lagrangian, i.e.\ on the validity of \eqref{eq: tonelli differential in x}. This condition guarantees that one use the dominated convergence theorem and differentiate under the integral sign.
\end{rem}
Next, as in the finite-dimensional case, by an argument following the lines of \cite[Eq.~(3.6), Eq.~(3.7)]{Abbo13Lect}, one can prove the following characterization of flow lines of the Euler–Lagrange flow of \( L \) in terms of the critical points of the action functional associated with \( L \).
\begin{lem}\label{lemma:ELE anc solutions}
Let $L:TG\to \RR$ be a strong Tonelli Lagrangian on $(G,\G)$, and let 
\( \gamma\in H^1([0,T],G)\). Then the following are equivalent:
\begin{enumerate}
    \item The curve \( \gamma \) is a solution of the Euler--Lagrange equations~\eqref{eq: ELE} of $L$ of energy $E=\kappa$; \label{itm:i}
    \item The curve \( \gamma \) is a critical point of the time-free action functional \( \SS_{L+\kappa} \). 
 \label{itm:ii}
\end{enumerate}
\end{lem}
We close this subsection by introducing, in the case where we drop the strong Tonelli assumption on \( L \), that is without imposing any condition on the dependence of the Lagrangian on the base point, a functional on \( L^2([0,1], T_e G) \times (0, \infty) \) which behaves better analytically. This construction is inspired by the work \cite{Bauer_2025} on the energy functional and it is going the starting point in our search for global minimizers later on.

\begin{lem}\label{lemma:regularity}
    Let \( L : TG \to \mathbb{R} \) be a Tonelli Lagrangian on \( (G, \mathcal{G}) \). Then the functional
    \begin{equation}\label{eq:func_on_tangent}
        S_{L+\kappa} \colon L^2([0,1], T_e G) \times (0, \infty) \to \mathbb{R}, \qquad
        (\xi, T) \mapsto T \int_0^1 \left( L\!\left(e, \tfrac{\xi(s)}{T}\right) + \kappa \right) \, \mathrm{d}s
    \end{equation}
    is of class \( C^1 \) and weakly lower semicontinuous.
\end{lem}
\begin{proof}
    The proof follows, after some adaptations, the argument of \cite[Prop.~3.1]{AbbondandoloSchwarz2009}. It is included for the sake of completeness in \Cref{s: proof of reg lemma}.
\end{proof}

We conclude by mentioning that, in the proof of \Cref{thm:existence of global minimizers}, we use the fact that, under some mild regularity assumptions on \( G \), which are always satisfied in the case of Sobolev diffeomorphism groups, the problem of minimizing the functionals in \eqref{eq: int action funct x,T} and \eqref{eq:func_on_tangent} are equivalent. This will be made precise after \Cref{Iprop: mini seques bounded from below and above} and proved thereafter.
\subsection{No-loss-no-gain.} 
We close this section by showing that the Euler--Lagrange flow preserves the regularity of elements in \( G \). This extends a recent result of Bauer, Harms and Michor to the setting of Tonelli Lagrangians. The authors show that the geodesic flow of a strong right-invariant metric on a half-Lie group \( G \) restricts to the subgroups \( G^\ell \) of \( C^\ell \)-elements in \( G \).
In particular, this means that there is neither gain nor loss of regularity along geodesics. 
We now aim to establish an analogous statement for the flow lines of the Euler--Lagrange flow of Tonelli Lagrangians and Tonelli Hamiltonians: 

\begin{prop}[No-loss–no-gain]\label{prop: no gain no loose}
Let $L$ be a Tonelli Lagrangian on \( (G, \mathcal{G}) \).
Then, for every \( \ell \geq 1 \), the Euler–Lagrange flow of $L$ defines a smooth map
\[
\Phi_t^{L} : TG^\ell \longrightarrow TG^\ell,
\]
for all \( t \) sufficiently small. 
In particular, the evolution along flow lines of $\Phi_t^{L}$ preserves regularity. 
The same conclusion also holds for the Hamiltonian flow of the Legendre dual \( H = E \circ \mathcal L^{-1} \) of $L$.
\end{prop}
\begin{proof}
By \Cref{lemm: ele define EL-flow}, solutions of the Euler–Lagrange equations are precisely the flow lines of the Hamiltonian flow of $E$, 
where $E$ is defined as in \eqref{eq: def tonelli energy} 
with respect to the pullback of the canonical symplectic structure via the Legendre transform to $TG$, 
that is \( \mathcal{L}^*\d\lambda_{\mathrm{can}} \) (the Hamiltonian vector field of \( E \) is well-defined because \( \mathcal L^* \d \lambda_{\mathrm{can}} \) is strong).
As \( L \) and \( \mathcal{G} \) are \( G \)-right-invariant, the same holds for \( E \) and \( \mathcal{L}^*\d\lambda_{\mathrm{can}} \). 
Hence, this Hamiltonian flow satisfies the assumptions of~\cite[Thm.~7.5]{Bauer_2025}, 
from which the claim for the Euler--Lagrange flow follows. 
This discussion, in combination with \Cref{lem: legendre transform conjugates ELE flow and Hamiltonian flow}, 
gives the statement in the case of the Tonelli Hamiltonian, which finishes the proof.
\end{proof}
\section{Mañé's critical values for Tonelli Lagrangians.} \label{s: Mañé's critical values for Tonelli Lagrangians.}
In this section, we begin with the variational definition of
\emph{Mañé's critical values}. 
This definition extends the variational formulation of Mañé’s critical value 
for magnetic Lagrangians on Hilbert manifolds, as introduced by the author 
in~\cite{M24}, to the setting of Tonelli systems.

We recall that Mañé~\cite{Man} assigned a critical value to each covering 
$\tilde{G} \to G$.  
In what follows, we concentrate on three significant coverings: 
the trivial covering $G$, the abelian covering $\tilde{G}_{\mathrm{ab}}$ 
associated with the commutator subgroup $[\pi_1(G), \pi_1(G)]$, 
and the universal covering $\hat{G}$. 
These correspond, respectively, to 
Mañé’s critical value, 
the \emph{strict} Mañé critical value, 
and the \emph{lowest} Mañé critical value 
for Tonelli Lagrangians.  
The construction parallels the ideas developed in the finite-dimensional setting~\cite{Man, CIPP}. 

\begin{defn}\label{Def: Manes value for Tonelli Langrangians} 
  Let $L : TG \to \mathbb{R}$ be a Tonelli Lagrangian on the half--Lie group $G$ 
with respect to the strong Riemannian metric $\mathcal{G}$. \\[0.5em]
Then the \emph{Mañé critical value} of \( L \) is defined by
\begin{align*}
    c(L) 
    &= \inf \Bigl\{ 
        \kappa \in \mathbb{R} \;\Big|\; 
        \SS_{L+\kappa}(x,T) \ge 0 
        \quad \forall\, x \in H^1([0,1], G),\; T > 0 
    \Bigr\} \\[0.4em]
    &= - \inf \Bigl\{ 
        \frac{1}{T} \int_0^T L(\gamma(t), \gamma'(t))\, dt 
        \;\Big|\; 
        \gamma \in H^1(\mathbb{R}/T\mathbb{Z}, G),\; T > 0 
    \Bigr\} .
\end{align*}

\noindent
The \emph{strict Mañé critical value} of \( L \) is defined by 
\begin{align*}
    c_0(L)
    :=\; &\inf \Bigl\{ 
        \kappa \in \mathbb{R} \;\Big|\; 
        \SS_{L+\kappa}(x,T) \ge 0 
        \quad \forall\, x \in H^1([0,1], G),\; T > 0,\;
        x \text{ null-homologous} 
    \Bigr\} \\[0.4em]
    =\;& - \inf \Bigl\{ 
        \frac{1}{T} \int_0^T L(\gamma(t), \gamma'(t))\, dt 
        \;\Big|\;
        \gamma \in H^1(\mathbb{R}/T\mathbb{Z}, G) \text{ null-homologous},\; T > 0 
    \Bigr\}.
\end{align*}

\noindent
The number
\begin{align*}
    c_u(L)
    :=\; &\inf \Bigl\{ 
        \kappa \in \mathbb{R} \;\Big|\; 
        \SS_{L+\kappa}(x,T) \ge 0 
        \quad \forall\, x \in H^1([0,1], G),\; T > 0,\;
        x \text{ contractible} 
    \Bigr\} \\[0.4em]
    =\;& - \inf \Bigl\{ 
        \frac{1}{T} \int_0^T L(\gamma(t), \gamma'(t))\, dt 
        \;\Big|\;
        \gamma \in H^1(\mathbb{R}/T\mathbb{Z}, G) \text{ contractible},\; T > 0 
    \Bigr\}
\end{align*}
is called the \emph{lowest Mañé critical value}.

\medskip

\noindent
Finally, the \emph{energy-level Mañé critical value} is defined by
\begin{align*}
    e_0(L)
    :=\; \max \bigl\{ 
        E(x,0) \;\big|\; x \in G 
    \bigr\}
    = \max \bigl\{ 
        E(x,v) \;\big|\; (x,v) \in \mathrm{Crit}(E) 
    \bigr\} .
\end{align*}

\end{defn}
\begin{rem}
    From \eqref{eq: energy electromagnetic} one sees that for magnetic Lagrangians the energy-level Mañé critical value $e_0(L)$ is always zero. 
    In contrast, for general Tonelli Lagrangians $L$ the value $e_0(L)$ need not vanish, as can also be seen by examples arising from \eqref{eq: energy electromagnetic}.
\end{rem}
\begin{rem}
   This extends the recent work of the first author on Mañé’s critical value on Hilbert manifolds in~\cite{M24} to Mañé’s critical values in the setting of Tonelli Lagrangians.
\end{rem}
We close this subsection by noting that, as in the finite-dimensional case (see~\cite{Abbo13Lect}), 
the following chain of inequalities between the Mañé critical values holds:
\[
    \min E \;\leq\; e_0(L) \;\leq\; c_u(L) \;\leq\; c_0(L) \;\leq\; c(L).
\]

\section{Existence of global minimizers and the Hopf--Rinow theorem.} \label{s: The Hopf–Rinow theorem for Tonelli Lagrangians.}
In this section, we prove the main results of this work, namely \Cref{prop: global existince of flow lines}, \Cref{thm:existence of global minimizers}, and \Cref{thm: Hopf-Rinow for Tonelli int}. 
We begin with the proof of \Cref{prop: global existince of flow lines}, then proceed to the technical core of this work, \Cref{thm:existence of global minimizers}, and finally conclude from \Cref{thm:existence of global minimizers} that \Cref{thm: Hopf-Rinow for Tonelli int} holds. 
\subsection{Proof of \Cref{prop: global existince of flow lines}}We adapt the proof strategy of Hopf-Rinow~\cite{HR} of geodesic completeness to the setting of Euler--Lagrange flows of Tonelli Lagrangians on half-Lie groups. More precisely, suppose a solution \( \gamma \) of the Euler--Lagrange equations is maximally defined on \( [0, T) \) for some \( T < \infty \) and pick an increasing sequence \( t_i \to T \). Note that, by \Cref{Ilem: Tonelli Langrangians} and \Cref{rem: fibrewise differential grows quadratically}, the energy
\[
    E(x, v) = D_v L(x, v)(v) - L(x, v)
\]
grows at least quadratically in the fibers, that is there exist constants \( \alpha, \beta > 0 \) such that 
\[
    E(x, v) \geq \beta \, \| v \|_{\mathcal{G}}^2 - \alpha \quad \forall (x, v) \in TG.
\]
This, together with the fact that the Euler–Lagrange flow preserves \( E \), implies that the speed \( \| \dot{\gamma} \| \) is uniformly bounded, i.e.\ there exists a constant \( C > 0 \) such that 
\begin{equation}\label{eq: velocity uniformly bounded}
    \sup_{t \in [0,T)} \| \dot{\gamma}(t) \|_{\mathcal{G}} < C.
\end{equation}
Since
\[
    d_{\mathcal{G}} \big( \gamma(t_i), \gamma(t_j) \big) 
    \leq \int_{t_i}^{t_j} \| \dot{\gamma}(t) \|_{\mathcal{G}} \, \mathrm{d}t 
    \leq C |t_j - t_i|
\]
for all \( t_i, t_j \in [0,T) \), the sequence \( \big( \gamma(t_i) \big)_{i \in \NN} \) is Cauchy in \( (G, \mathcal{G}) \).  
As \( G \), equipped with the geodesic distance \( d_{\mathcal{G}} \) induced by \( \mathcal{G} \), is metrically complete \cite[Theorem~7.7]{Bauer_2025}, the sequence \( \big( \gamma(t_i) \big)_{i \in \mathbb{N}} \) converges. 
Together with the fact tha we may also assume that the sequence \( \big( \| \dot{\gamma}(t_i) \| \big)_{i \in \mathbb{N}} \) converges (because it is bounded), this implies that \( \gamma \) can be extended to \( [0, T + \varepsilon) \) for some sufficiently small \( \varepsilon > 0 \), which contradicts maximality of \( T \), and thus completes the first part of the proof.

Finally, completeness of the Euler--Lagrange flow of \( L \) restricted to \( G^\ell \) follows directly from the no-loss-no-gain result (\Cref{prop: no gain no loose}), in combination with the discussion above. \qed 

\subsection{Proof of \Cref{thm:existence of global minimizers}}
For points \( p, q \in G \), we denote by \( H^1_{p, q} \) the space of \( H^1 \) curves \( x \colon [0, 1] \to G \) connecting \( p\) and \( q \). Recall that a pair \( (x, T) \in H^1_{p, q} \times (0, \infty) \) is a \textit{global minimizer} of \( \mathbb S_{L+ \kappa} \) if
\begin{equation}
    \mathbb S_{L + \kappa}(x, T) = \inf_{H^1_{p, q} \times (0, \infty)} \mathbb S_{L + \kappa}.
\end{equation}

\paragraph{\textbf{Step 1: the minimizing sequence.}}
Fix \( p, q \in G \). Note that for energies above the 
lowest Mañé critical value, the action functional $\mathbb S_{L+\kappa}$ is uniformly bounded from below on $\mathcal{P}(p, q)$, as the following lemma shows.

\begin{lem}\label{Ilemm: uniform lower bound of Act functional above manes value}
    Fix $p, q \in G$. If $\kappa \geq c_u(L)$, then there exists a constant $C(p,q) > 0$ such that
    \begin{equation*}
        \SS_{L+\kappa}(\gamma) \ge -C(p,q) 
        \qquad \forall\, \gamma \in \mathcal{P}(p,q).
    \end{equation*}
\end{lem}
\begin{proof}
    Let \( \gamma \in \mathcal{P}(p, q) \) and lift it to a path \( \tilde \gamma \colon [0, T] \to \tilde G \) in the universal cover \( \pi \colon \tilde G \to G \). Let now \( \tilde \alpha \) be any path going from \( \tilde \gamma(T) \) to \( \tilde \gamma (0) \). Then, the concatenation \( \tilde \gamma * \tilde \alpha \) is a contractible loop. Hence, so is its projection \( \pi (\tilde \gamma * \tilde \alpha) \). Since \( \kappa \geq c_u(L) \), we get
    \begin{equation*}
        0 \leq \SS_{L + \kappa} \big( \pi (\tilde \gamma * \tilde \alpha) \big) = \SS_{L + \kappa}(\gamma) + \SS_{L + \kappa}(\pi \tilde \alpha)
    \end{equation*}
    and the statement follows with \( C(p, q) = - \mathbb S_{L+ \kappa} (\pi \tilde \alpha) \).
\end{proof}

As a consequence of \Cref{Ilemm: uniform lower bound of Act functional above manes value}, 
for given $p,q \in G$ and $\kappa > c_u(L)$ the infimum of the action functional $\mathbb S_{L+\kappa}$ 
over $\mathcal{P}(p,q)$ is finite, i.e.
\begin{equation}\label{eq: definition inf of Slk}
    C_{p,q} := \inf_{\gamma \in \mathcal{P}(p,q)} \SS_{L+\kappa}(\gamma) \in (-\infty, \infty).   
\end{equation}
Therefore, there exists a minimizing sequence $(\gamma_n) \subseteq \mathcal{P}(p,q)$, 
where each $\gamma_n$ is defined on \( [0, T_n] \) for some \( T_n > 0 \), such that
\[
    \bigl| C_{p,q} - \SS_{L+\kappa}(\gamma_n) \bigr| \xrightarrow{n \to \infty} 0.
\] 
This can be reformulated by setting  
\[
    x_n \colon [0, 1] \to G, 
    \qquad 
    x_n(s) := \gamma_n(sT_n),
\]  
and using~\eqref{eq: int action funct x,T}, which yields
\begin{equation}\label{eq: minimizing sequence tonelli action functional}
    \bigl| C_{p,q} -\SS_{L+\kappa}(x_n,T_n) \bigr| \xrightarrow{n \to \infty}0.
\end{equation}
Hence, we obtain a minimizing sequence $[(x_n,T_n)]$ for $\mathbb S_{L+\kappa}$ in \( H^1_{p, q} \times (0, \infty) \). In order to use $[(x_n,T_n)]$ to find a minimizer of \( \mathbb S_{L+\kappa} \), we must first exclude the possibilities
\begin{equation}\label{eq: periods collapse blowup}
    T_n \searrow 0 
    \qquad \text{or} \qquad 
    T_n \nearrow \infty 
    \quad \text{as } n \to \infty .
\end{equation}
We begin by excluding the first possibility in \eqref{eq: periods collapse blowup} by:
\begin{lem}\label{Lem: Tn bounded away from zero}
        The sequence \( (T_n) \) is bounded away from \( 0 \).
\end{lem}
\begin{proof}
     By \Cref{Ilem: Tonelli Langrangians}, Tonelli Lagrangians grow at least quadratically and, using~\eqref{eq: int action funct x,T}, one obtains
    \[
       \SS_{L+\kappa}(x_n,T_n)\;\geq\; 
        \frac{m}{4T_n} \Vert \dot{x}_n \Vert^2_{L^2} - b T_n .
    \]
    Note that the \( L^2 \)-norm of \( \dot{x}_n \) is bounded away from zero (unless \( p = q \)), since \( \| \dot x_n \|_{L^2} \geq \| \dot x_n \|_{L^1} \) and the latter is at least the distance between \( p \) and \( q \), which is non zero as the geodesic distance $d_\G$ on $G$ is nondegenerate by \cite[Thm.\ 7.3]{Bauer_2025}.
    In combination with~\eqref{eq: minimizing sequence tonelli action functional}, 
this shows that $(T_n)$ cannot converge to zero.   
\end{proof}

To exclude the possibility \( T_n \to \infty \) in \eqref{eq: periods collapse blowup}, we first state the following result.
\begin{lem}\label{lem: loops action bounded by period from below}
    Let $\kappa > c(L)$. Then there exists $\varepsilon > 0$ such that for all \( H^1 \) loops $x \colon \mathbb{R} / T \mathbb{Z} \to G$,
    \[
       \SS_{L+\kappa}(x,T) \geq \varepsilon T.
    \]
\end{lem}
\begin{proof}
    Suppose this is not the case. Then, for all \( n \in \mathbb{N} \) there exists a loop \( x_n \) of period \( T_n \) such that \( \mathbb S_{L+ \kappa}(x_n, T_n) < \frac{T_n}{n}\). 
    Let now \( \bar \kappa \in \big( c(L), \kappa \big) \) and note that
    \begin{equation*}
        \mathbb S_{L + \bar \kappa} (x_n, T_n) = \SS_{L + \kappa} (x_n, T_n) + (\bar \kappa - \kappa) T_n < \bigg( \frac{1}{n} + \bar \kappa - \kappa \bigg) T_n.
    \end{equation*}
    Since \( \bar \kappa < \kappa \), for \( n \) large enough the right-hand side of the above equation is negative, which contradicts \( \bar \kappa > c(L) \). This finishes the proof.
\end{proof}
The above lemma readily implies the following.
\begin{lem}\label{lem: period bounded from above}
     If \( \kappa > c(L) \), the sequence \( (T_n) \) is bounded from above.
\end{lem}
\begin{proof}
    Fix a path \( \alpha \colon [0, T_\alpha] \to G \) going from \( q \) to \( p \). Then, since we chose \( \kappa > c(L) \), \Cref{lem: loops action bounded by period from below} implies that
    \begin{equation*}
        \SS_{L + \kappa} (x_n, T_n) + \SS_{L + \kappa}(\alpha, T_\alpha) = \SS_{L + \kappa}(x_n * \alpha, T_n + T_\alpha) \geq \varepsilon (T_n + T_\alpha).
    \end{equation*}
    As the sequence \( \big( \SS_{L + \kappa}(x_n, T_n) \big) \) is bounded by \eqref{eq: minimizing sequence tonelli action functional}, it follows that \( T_n \) cannot diverge to \( \infty \).
\end{proof}
Combining \Cref{Lem: Tn bounded away from zero} and \Cref{lem: period bounded from above}, we have excluded both possibilities in \eqref{eq: periods collapse blowup}. Thus, we have proven:
\begin{prop}\label{Iprop: mini seques bounded from below and above}
    Let \( \kappa > c(L) \) let and $ \big( (x_n,T_n) \big) \subseteq H^1_{p,q}\times (0,\infty)$ be a minimizing sequence as in~\eqref{eq: minimizing sequence tonelli action functional}. 
    Then there exist constants $0<a<b<\infty$ such that
    \[
        T_n \in [a,b] \qquad \forall n \in \NN.
    \]
\end{prop}

Before moving on, we reformulate our minimization problem in terms of \( L^2 \)-paths in \( T_e G \), analogously to the treatment of the energy functional in \cite[Theorem~7.7]{Bauer_2025}. The idea behind this is that the functional \( S_{L + \kappa} \) is analytically easier to handle than \( \mathbb{S}_{L + \kappa} \) (cf.\ \Cref{lemma:regularity}). 
To this end, note that since \( L \) is \( G \)-right-invariant, we have
\begin{equation*} 
    \begin{aligned}
        \mathbb S_{L+\k}(x,T) &= T \int_0^1 L\left(x, \frac{\dot x}{T}\right)+\k \, \d s \\
        &= T \int_0^1 L\left(e, \frac{\mathrm{d}_{x} \mu^{x^{-1}} (\dot{x})}{T}\right)+\k\, \d s = S_{L + \kappa} \big( \mathrm{d}_{x} \mu^{x^{-1}} (\dot{x}), T \big).
    \end{aligned}
\end{equation*}
The \( L^2 \)-regularity of \( G \), i.e.\ the bijection between \( H^1 \)-paths in \( G \) and \( L^2 \)-paths in \( T_e G \), implies that minimizing \( \mathbb S_{L+\k} \) over \( H^1_{p,q}\times(0,\infty) \) is equivalent to the minimizing \( S_{L + \kappa} \) over \( A_{qp^{-1}} \times (0, \infty) \), that is:
\begin{equation}\label{eq: min of S and mathbb S match}
     C_{p, q} = \inf_{(x, T) \in \mathcal P(p, q)} \mathbb S_{L+\k}(x, T) =
    \inf_{\xi \in A_{qp^{-1}} \times (0, \infty)} S_{L + \kappa}(\xi, T)\, .
\end{equation}
The minimizing sequence we constructed above then corresponds to a sequence 
\begin{equation}\label{eq:def minimizing sequence xin}
    (\xi_n, T_n) \in A_{qp^{-1}} \times [a, b]
\end{equation}
satisfying \( \bigl| C_{p,q} - S_{L+\kappa}(\xi_n,T_n) \bigr| \to 0 \) as \( n \to \infty \).

\medskip

\paragraph{\textbf{Step 2: convergence of the minimizing sequence.}} We now prove that the sequence \( (\xi_n, T_n) \) converges (weakly) to a minimizer of \( S_{L+\kappa} \). 
We adapt the argument from the proof of \cite[Theorem~7.7]{Bauer_2025} for the energy functional to our setting.

Since Tonelli Lagrangians grow at least quadratically by \Cref{Ilem: Tonelli Langrangians} and  by \Cref{Iprop: mini seques bounded from below and above} it holds \( T_n \geq a > 0 \) for all \( n \), the sequence \( (\xi_n) \) in~\eqref{eq:def minimizing sequence xin} is \( L^2\)-bounded. \\
Thus, by the Eberlein-Smulyan theorem, we can assume (up to extracting a subsequence) that 
\[ \xi_n \to \xi \in L^2([0, 1], T_eG)\]
weakly, as well as \( T_n \to T \in [a, b] \). Since \( \xi_n \in A_{qp^{-1}} \) for all \( n \) and \( A_{qp^{-1}} \) is weakly closed by assumption, the weak limit \( \xi \) also lies in \( A_{qp^{-1}} \).

The weak lower semicontinuity of \( S_{L + \kappa} \), established in \cref{lemma:regularity}, implies that
\[
    S_{L+\k}(\xi,T)\leq \liminf_{n \to \infty} S_{L+\k}(\xi_n, T_n).
\]
From this in combination with \eqref{eq: min of S and mathbb S match}  we can conclude that $(x=\mathrm{Evol}(\xi),T)$ is a minimizer of $\mathbb S_{L+\kappa}$ over $\mathcal P(p, q)$. This finishes the proof of \Cref{thm:existence of global minimizers}. \qed

\subsection{Proof of \Cref{thm: Hopf-Rinow for Tonelli int}} Let us prove our third main result. 
We begin by noting that for all \( p, q \in G \), the space \( H^1_{p, q} \) of curves connecting them is a smooth submanifold of \( H^1([0, 1], G) \), since the evaluation map
\begin{equation}
    H^1([0, 1], G) \to G \times G, \qquad x \mapsto \big( x(0), x(1) \big)
\end{equation}
is a smooth submersion.

Let \( (x, T) \in H^1_{p,q} \) be the global minimizer of \( \mathbb{S}_{L+\kappa} \) connecting \( p \) and \( q \) provided by \Cref{thm:existence of global minimizers}. In combination with the fact that, if \( L \) is a strong Tonelli Lagrangian, the functional \( \mathbb{S}_{L+\kappa} \) is of class \( C^1 \) by \Cref{lem: strong tonelli implies action is c1}, this implies that \( (x, T) \) is a critical point of
\[
    \mathbb{S}_{L+\kappa} \colon H^1_{p,q} \times (0, \infty) \longrightarrow \mathbb{R}.
\]
By \Cref{lemma:ELE anc solutions}, the critical point \( (x, T) \) of \( \mathbb{S}_{L+\kappa} \) corresponds to a flow line of the Euler--Lagrange flow of \( L \) with energy \( \kappa \). This finishes the proof.\qed

\appendix
\section{Proof of \Cref{lemma:regularity}}\label{s: proof of reg lemma}
    We adapt the proof of \cite[Prop.\ 3.1]{AbbondandoloSchwarz2009}. Let \( \delta > 0 \) and \( (\xi, T) \in L^2([0, 1], T_eG) \times (0, \infty) \). For every tangent vector \( (\eta, r) \) at \( (\xi, T) \) we have
    \begin{equation*}
        \begin{aligned}
            &S_ {L + \kappa} (\xi + \delta \eta, T + \delta r) - S_{L + \kappa}(\xi, T) \\
            &= \delta \kappa r + (T + \delta r) \int_0^1 L\bigg( e, \frac{\xi + \delta \eta}{T + \delta r} \bigg) \mathop{}\!\mathrm{d}t - T \int_0^1 L\bigg( e, \frac{\xi}{T} \bigg) \mathop{}\!\mathrm{d}t \\
            &= \delta \kappa r + \delta r \int_0^1 L\bigg( e, \frac{\xi}{T} \bigg) \mathop{}\!\mathrm{d}t + (T + \delta r) \int_0^1 L\bigg( e, \frac{\xi + \delta \eta}{T + \delta r} \bigg) \mathop{}\!\mathrm{d}t - (T + \delta r) \int_0^1 L\bigg( e, \frac{\xi}{T} \bigg) \mathop{}\!\mathrm{d}t \\
            &= \frac{\delta r S_{L + \kappa}(\xi, T)}{T} + (T + \delta r) \int_0^1 \int_0^1 D_vL \bigg(e, \frac{\xi + \delta s \eta}{T + s \delta r} \bigg) \bigg( -\frac{\delta r(\xi + s \delta \eta)}{(T + s \delta r)^2}  + \frac{\delta \eta}{T + s \delta r} \bigg) \mathop{}\!\mathrm{d}s \mathop{}\!\mathrm{d}t.
        \end{aligned}
    \end{equation*}
    Then, the Tonelli assumptions and the bounded convergence theorem imply that
    \begin{equation*}
        \frac{S_{L + \kappa}(\xi + \delta \eta, T + \delta r) - S_{L + \kappa}(\xi, T)}{\delta}
    \end{equation*}
    converges to 
    \begin{equation}\label{eq:diff_action}
        \mathrm d S(\xi, T)(\eta, r) := \frac{r S_{L + \kappa}(\xi, T)}{T} + T \int_0^1 D_vL \bigg( e, \frac{\xi}{T} \bigg) \bigg( -\frac{r\xi}{T^2} + \frac{\eta}{T} \bigg) \mathop{}\!\mathrm{d}t.
    \end{equation}
    as \( \delta \to 0 \).
    A straightforward application of the H\"older inequality together with the Tonelli assumptions shows that \( \mathrm d S(\xi, T) \) is a bounded linear operator on \( L^2([0, 1], T_eG) \times \mathbb R \). Thus, \( S_{L + \kappa} \) is \( C^1 \) if we can show that \( \mathrm d S (\xi, T) \) depends continuously on \( (\xi, T) \). This follows readily from the Tonelli assumptions and the bounded convergence theorem.

    Let us prove weak lower semicontinuity. Note that \( L \) is by definition convex. Fix now \( (\xi_0, T_0), (\xi_1, T_1) \in L^2([0, 1], T_eG) \in [a, b] \) and set \( (\xi_{\lambda}, T_\lambda) = \big((1 - \lambda) \xi_0 + \lambda \xi_1, (1 - \lambda) T_0 + \lambda T_1 \big) \) for all \( \lambda \in [0,1] \). Then, we have
\begin{equation*}
    \begin{aligned}
        T_{\lambda} \int_0^1 L \bigg(e, \frac{\xi_{\lambda}}{T_{\lambda}} \bigg) \mathop{}\!\mathrm{d} s &= T_{\lambda} \int_0^1 L \bigg( e, \frac{(1-\lambda) \xi_0 + \lambda \xi_1}{T_{\lambda}} \bigg) \mathop{}\!\mathrm{d} s \\
        &= T_{\lambda} \int_0^1 L \bigg( e, \frac{(1- \lambda) T_0 }{T_\lambda}\frac{\xi_0}{T_{0}} + \frac{\lambda T_1}{T_\lambda} \frac{\xi_1}{T_1} \bigg) \mathop{}\!\mathrm{d} s \\
        &\leq T_{\lambda} \int_0^1 \frac{(1 - \lambda) T_0}{T_{\lambda}} L \bigg( e, \frac{\xi_0}{T_0} \bigg) \mathop{}\!\mathrm{d} s + T_{\lambda} \int_0^1 \frac{\lambda T_1}{T_{\lambda}} L \bigg( e, \frac{\xi_1}{T_1} \bigg) \mathop{}\!\mathrm{d} s \\
        &= (1 - \lambda) T_0 \int_0^1 L \bigg( e, \frac{\xi_0}{T_0} \bigg) \mathop{}\!\mathrm{d} s + \lambda T_1 \int_0^1 L \bigg( e, \frac{\xi_1}{T_1} \bigg) \mathop{}\!\mathrm{d} s,
    \end{aligned}
\end{equation*}
which shows that \( S_{L + \kappa} \) is convex. The claim then follows from the standard fact that the closure of a convex set coincides with its weak closure.\qed
\bibliographystyle{abbrv}
	\bibliography{ref}
\end{document}